\documentclass[12pt]{article}

\usepackage{amsmath}
\usepackage{amsfonts,graphicx}
\usepackage{amssymb,amsthm}
\usepackage[latin1]{inputenc}

\def\R{\mathbb{R}}
\def\N{\mathbb{N}}

\theoremstyle{plain}
\newtheorem{theorem}{Theorem}
\newtheorem{Prop}{Proposition}

\newtheorem{lemma}{Lemma}

\theoremstyle{remark}

\begin{document}

\title{Quantitative stability of certain families of periodic solutions in the Sitnikov problem}
\author{Jorge Gal\'an \footnote{Universidad de Sevilla, Sevilla- Spain.} \and Daniel N\'u\~nez \footnote{Pontificia Universidad Javeriana-Cali, Colombia.} \and Andr\'es Rivera \footnote{Pontificia Universidad Javeriana-Cali, Colombia.} }
\maketitle



\section{Introduction} 
There are a extensive research devoted to the study of periodic solutions in the \textit{the Sitnikov problem}, which is defined as follows: Two bodies with equal mass $m_{1}= m_{2}$ (\textit{called primaries}) are moving in the plane $x,y$ around their center of mass (\textit{barycenter}) as solutions of the planar two body problem, a third body $m_{3}$ with zero mass move along $z$-axis through the barycenter of the primaries. The Sitnikov problem deals with the study of the orbits of $m_{3}.$ In appropriate units the equation of motion of the zero mass body is
\begin{equation}\label{Sit}
\ddot{z}=-\frac{z}{(z^{2}+r(t,e)^{2})^{3/2}},
\end{equation}

\noindent
where $e\in \,[0,1[$ is the eccentricity of the elliptic orbits described by the primaries and $r(t,e)$ denotes the distance from the primaries to the origin (center of mass). The function $r(\cdot,e)$ has minimal period $2\pi$ and is implicitly defined in terms of Kepler's equation, namely
\begin{equation}\label{kepler}
r=\frac{1}{2}(1-e\cos u), \quad u-e\sin u =t.
\end{equation}

Many contributions has been given about the dynamics in the Sitnikov problem both from the analytical and numerical point of view, since its formulation by K.A. Sitnikov in 1960. We refer to \cite{Alekseev
II,Moser} for the most classical results and \cite{Martinez-Chiralt} for numerical results. Since the Sitnikov equation is a forced oscillator with minimal period $2\pi$ (for $e \neq 0$)  one of the first questions is the study of
families of periodic solutions which depend continuously on the
eccentricity. It can be proved that the period of this families must be equal to $2N\pi$ for some $N\in \N$, see \cite{Rivera2}. We call this solutions \textit{subharmonics}. The searching of \textit{subhamonics} became more simpler if one is restricted
to the symmetric case: even or odd solutions.
Notice that the function $r(\cdot,e)$ is even and so (\ref{Sit}) is
invariant under the  symmetries
$$(t,z)\mapsto (-t,z),\qquad (t,z)\mapsto (t,-z),$$
so one can obtained for all $N \in\N$ an even $2N\pi $-periodic solution by solving the boundary value problem
\begin{equation}\label{bvp}
\ddot{z}=-\frac{z}{(z^{2}+r(t,e)^{2})^{3/2}},\qquad \dot{z}(0)=\dot{z}(N\pi)=0,
\end{equation}
and by extending symmetrically on the interval $[-N\pi, 0]$ and finally exten\-ding periodically over all $\R$. This approach is called the \textit{shooting method}: the searching the suitable initial position $\xi=z(0)$ for each $e \in [0,1[$  from the rest in order to obtain the second boundary condition in (\ref{bvp}). In this way $\xi=\xi(e)$ will be a continuous function for small values of $e$, provided local families of periodic solutions parametrized by the eccentricity. Using the method of global continuation of Leray-Schauder, Llibre and Ortega in \cite{Llibre-Ortega} proved that these families can be continued from the known $2N\pi$-periodic solutions in the circular case ($e=0$) for nonnecessarily small values of the eccentricity $e$ and in some cases for all values of $e\in \, [0,1[.$ However this approach does not say anything about the stability properties of this periodic solutions.

\vspace{0.5 cm}
\noindent
It is well known that for $e=0$ there are a finite number of nontrivial subharmoncis (with period $2N\pi$). On the other hand all them are parabolic and unstable (in the Lyapunov sense) if we consider the unperturbed autonomous equation ($e=0$) like a $2\pi$-periodic equation.

\noindent
In this document we present a new method that quantifies the mentioned bifurcating families and them stabilities properties at least in first approximation. Our approach proposes two general methods: The first one is to estimate the growing of the canonical solutions for one-parametric  differential equation of the form
\[
\ddot{x}+a(t,\lambda)x=0,
\]
with $a\in C^{1}([0,T] \times [0,\Lambda])$ (Lemma \ref{control hill}, Section 2). The second one gives stability criteria for one-parametric Hill's equation of the form
\[
\ddot{x}+q(t,\lambda)x=0, \quad (\ast)
\]
where $q(\cdot,\lambda)$ is $T$-periodic and $q\in C^{3}(\R\times [0,\Lambda])$, such that for $\lambda=0$ the equation $(*)$ is parabolic (Lemma \ref{LH}, Section 2). The Lemma \ref{LH} determines an explicit $\lambda$-interval of ellipticity of hiperbolicity for $(\ast)$. Henceforth this can be viewed as a quantified version of stability classical results for Hill's equation like in \cite{Magnus-Winkler}. To sum up, the main contributions of this document besides of the two mentioned before are the following:

\begin{enumerate}
\item For any $N\in \N$ odd, we gives sufficient conditions for the ellipticity of hyperbolicity of the families of nontrivial even, $2N\pi$-periodic solutions of (\ref{Sit}) in a computable interval of eccentricities $e$ (Theorem, Section 4 ).

\item  For $N=1,3$ we shows  that all  families of nontrivial even, $2N\pi$-periodic solutions of (\ref{Sit}) are elliptic for $e\in ]0,e^{*}[$ for a computable $e^{*}$ (Section 4 and Section 5).
\end{enumerate}

\section{Fundamental results}
In this part of the document we introduce, to the best of our knowledge, a novel technique to estimate uniform bounds for the growing of the canonical solutions of second order differential equations  of the form
\begin{equation}\label{second order}
\ddot{y}+a(t,\lambda)y=0, \qquad (*)
\end{equation}
for $\displaystyle{a(t,\lambda)\in C^{1}([0,T]\times [0,\Lambda])}$ based on the zeros of an appropriate function. Notice that in the particular case $a(t+T,\lambda)=a(t,\lambda)$ for all $(t,\lambda)\in \R\times [0,\Lambda]$ and $T>0$ we have a parametric Hill's equation.


\begin{lemma}\label{control hill} Consider the family of  equations $(\ref{second order})$ with the previous hypo\-thesis on the function $a(t,\lambda)$. Suppose that $\phi_{1}(t,\lambda),\phi_{2}(t,\lambda)$ are the canonical solutions of $(\ref{second order})$, i.e. 
\[
\phi_{1}(0,\lambda)=\dot{\phi}_{2}(0,\lambda)=1, \quad \dot{\phi}_{1}(0,\lambda)=\phi_{2}(0,\lambda)=0,
\]
for all $\lambda\in [0,\Lambda]$. For each $\lambda$ define
\begin{equation}\label{R}
R_{\lambda}:=\sup_{\mu\in [0,\lambda]}\max \left\{\|\phi_{i}(\cdot\,,\mu)\|_{\infty}, \, \|\dot{\phi}_{i}(\cdot\,,\mu)\|_{\infty} \,: i=1,2\right\}
\end{equation}
Let $r_{0}$ a positive number greater than $R_{0}$. Assume the following conditions
\begin{enumerate}
\item Exist a positive continuous function $d=d(\lambda,R)$ which is increasing in both variables and such that 
\[
d(\lambda,R_{\lambda})\geq \sup_{t\in [0,T]}|\partial_{\lambda}a(t;\lambda)|
\]
\item There is a $\lambda^{*}\in  \,]0,\Lambda]$ such that the function $Q(\lambda,\cdot)$ has at least two consecutive zeros for each $\lambda \in [0,\lambda^{*}]$ where 
\[
Q(\lambda,R)=2T\lambda R^{3} d(\lambda,R)-R+r_{0},
\]
and moreover if $R_{1,\lambda}, R_{2,\lambda}$ are the first two consecutive zeros of  $\displaystyle{Q(\lambda,\cdot)}$ then 
\begin{enumerate}
\item $r_{0} \leq R_{1,\lambda}$ for all $\lambda \in [0,\lambda^{*}]$ 
\item $Q(\lambda,R)>0\,(<0)$  for all $R\in [0,R_{1,\lambda}[$ ($R\in [R_{1,\lambda},R_{2,\lambda}[$) and for all $\lambda \in [0,\lambda^{*}]$.
\end{enumerate}
Then $R_{\lambda}\leq R_{1,\lambda}$ for all $\lambda \in ]0,\lambda^{*}]$.
\end{enumerate}
\end{lemma}
 

\vspace{0.5 cm}
\begin{proof}
For a fixed $\lambda \in [0,\Lambda]$ and $\mu \in [0,\lambda]$, let $\phi_{i}(\cdot,\mu)$, $i=1,2$ the canonical solutions of $(\ref{second order})$. By the Mean Value Theorem we obtain 
\[
\big|\phi_{i}(t,\mu)-\phi_{i}(t,0)\big| \leq \max_{\mu \in [0,\Lambda]} \|\partial_{\mu}\phi_{i}(\cdot\,,\mu)\|_{\infty}\,\mu,
\]
\noindent
for all $t\in [0,T]$. By differentiability respect to parameters, the function $y_{i}(t)=\partial_{\mu}\phi_{i}(t,\mu)$ satisfies the following Cauchy problem 
\begin{equation}\label{cp}
\left\{
\begin{aligned}
\ddot{y}+a(t;\mu)y&=-\partial_{\mu}a(t;\mu)\phi_{i}(t,\mu)\\
y(0)=\dot{y}(0)&=0,
\end{aligned}
\right.
\end{equation}
Therefore by the method of variation of parameters we obtain
\begin{equation*}
\begin{split}
|y_{i}(t)|&\leq \int_{0}^{t}\Big|G(s,t,\mu)\partial_{\mu}a(s;\mu)\phi_{i}(s,\mu)\Big|ds,\\
|\dot{y}_{i}(t)|&\leq \int_{0}^{t}\Big|\partial_{t}G(s,t,\mu)\partial_{\mu}a(s;\mu)\phi_{i}(s,\mu)\Big|ds,
\end{split}
\end{equation*} 
where $\displaystyle{G}(t,s,\mu)=\phi_{1}(s;\mu)\phi_{2}(t,\mu)-\phi_{1}(t,\mu)\phi_{2}(s,\mu)$. In consequence, for all $t\in [0,T]$ we have 
\[
|y_{i}(t)|, |\dot{y}_{i}(t)|\leq 2TR_{\mu}^{3}d(\mu,R_{\mu}).
\]
The above inequalities implies 
\begin{equation*}
\begin{split}
\big|\phi_{i}(t,\mu)\big|&\leq \big|\phi_{i}(t,0)\big|+2T\mu R_{\mu}^{3}d(\mu,R_{\mu}), \\
\big|\dot{\phi}_{i}(t,\mu)\big|&\leq \big|\dot{\phi}_{i}(t,0)\big|+2T\mu R_{\mu}^{3}d(\mu,R_{\mu}).
\end{split}
\end{equation*} 

\noindent
From the monotonicity of $d(\lambda, R)$ we deduce
\[
\|\phi_{i}(t,\mu)\|_{\infty}, \|\dot{\phi}_{i}(t,\mu)\|_{\infty}\leq r_{0}+2T\lambda R_{\lambda}^{3}d(\lambda,R_{\lambda})
\]
for all $\mu \in [0,\lambda]$ and $\lambda \in [0,\lambda^{*}]$. Therefore, $\displaystyle{Q(\lambda,R_{\lambda})\geq 0}$ for all $\lambda \in [0,\lambda^{*}]$. Then, from the assumption \textit{2}. (part (b)) we have that 
\[
R_{\lambda}\in [0,R_{1,\lambda}] \cup [R_{2,\lambda},\infty[
\]
By continuity the set $\displaystyle{\left\{R_{\lambda}: \lambda \in [0,\lambda_{*}]\right\}}$ is an interval. Besides, $\lim_{\lambda \to 0}R_{\lambda}=R_{0}$ this implies that $R_{\lambda}\in [0,R_{1,\lambda}]$ for all  $\lambda \in [0,\lambda^{*}]$. This completes the proof.
\end{proof}

\vspace{0.5 cm}
\noindent
In Lemma \ref{LH} we present a simple quantified stability criteria for parametric Hill's equation of the form
\begin{equation}\label{Hequation}
\ddot{x}+q(t,\lambda)x=0,
\end{equation}
with $q\in C(\R\times [0,\Lambda])$ and $T$-periodic in $t$, when $|\Delta(0)|=2$, with $\Delta(\lambda)$ is the discriminant function, defined  as the trace of a monodromy matrix for the associated first order system to (\ref{Hequation}).

\begin{lemma}\label{LH} Consider the Hill's equation (\ref{Hequation})
where $q\in C^{3}(\R\times [0,\Lambda])$ and $T$-periodic in $t$. Let $\Delta(\lambda)$ the discriminant function of $(\ref{Hequation})$ that satisfies
\[
\Delta(0)=2, \quad \Delta^{\prime}(0)=0, \quad \text{and}\quad \Delta^{\prime \prime}(0)\neq 0.
\]
Define
\[
p(\lambda)=\mathcal{K}\lambda^{3}-3\Delta^{\prime \prime}(0)\lambda^{2}-24, \quad \text{and} \quad \mu=\frac{3|\Delta^{\prime \prime}(0)|}{\mathcal{K}},
\]
where $\mathcal{K}$ a positive constant such that $\displaystyle{\mathcal{K}\geq \sup_{\lambda\in [0,\Lambda]}|\Delta^{\prime \prime \prime}(\lambda)|}$.
\begin{enumerate}
\item [i)] If $\Delta^{\prime \prime}(0)>0$ then $\Delta(\lambda)>2$ for all $\lambda \in I_{1}=]0,\min\left\{\mu,\Lambda\right\}[$,
\item [ii)] If $\Delta^{\prime \prime}(0)<0$  and $p(\Lambda)>0$ (resp. $p(\Lambda)\leq 0$) then $|\Delta(\lambda)|<2$ for all $\lambda \in I_{2}=]0,\min\left\{\mu,\mu_{0},\Lambda\right\}[$ (resp. $\lambda \in I_{1}$) where $\displaystyle{\mu_{0}}$ is the unique positive root of $\displaystyle{p(\lambda)=0}$.
\end{enumerate}

\end{lemma}


\begin{proof}
For the classical Taylor's expansion over $\Delta(\lambda)$ in $[0,\Lambda]$ we have
\begin{equation}\label{taylor-2}
\Delta(\lambda)=2+\frac{\Delta^{\prime \prime}(0)\lambda^{2}}{2}+R(\lambda), \quad \forall \lambda\in [0,\Lambda],
\end{equation}
where $R(\lambda)$ is the remainder bounded by 
\[
|R(\lambda)|\leq \frac{\lambda^{3}\mathcal{K}}{3!}.
\]
Assume that $\Delta^{\prime \prime}(0)>0$ then  
\[
\frac{\Delta^{\prime \prime}(0)\lambda^{2}}{2}-\frac{\mathcal{K}\lambda^{3}}{3!}>0, \quad \text{if} \quad 0<\lambda<\mu.
\]

\noindent
In consequence, from (\ref{taylor-2}) and the estimative over $R(\lambda)$, we obtain
\begin{equation}
\Delta(\lambda)-2=\frac{\Delta^{\prime \prime}(0)\lambda^{2}}{2}+R(\lambda)>\frac{\Delta^{\prime \prime}(0)\lambda^{2}}{2}-\frac{\mathcal{K}\lambda^{3}}{3!}>0,
\end{equation}
if $\lambda \in I_{1}$, proving $i)$. 

\vspace{0.5 cm}
\noindent
Now, we suppose $\Delta^{\prime \prime}(0)<0$. Notice that $|\Delta(\lambda)|<2$ is equivalent to 
\[
-4<\Delta(\lambda)-2<0.
\]
Therefore, it is sufficient solve the following system of inequalities
\begin{equation}\label{desigualdades}
\begin{split}
\frac{\Delta^{\prime \prime}(0)\lambda^{2}}{2}+\frac{\mathcal{K}\lambda^{3}}{3!}&< 0,\\
\frac{\mathcal{K}\lambda^{3}}{3!}-\frac{\Delta^{\prime \prime}(0)\lambda^{2}}{2}&< 4,
\end{split}
\end{equation}
for $\lambda \in [0,\Lambda]$. The first inequality in (\ref{desigualdades}) is equivalent to $\lambda\in I_{1}$. The second one can be rewritten as $p(\lambda)<0.$  Notice that $p(\lambda)$ is a strictly increasing function for all $\lambda \in ]0,\Lambda]$. Then, if $p(\Lambda)\leq 0$ the second inequality holds for $\lambda \in I_{1}.$ Else, $p(\Lambda)>0$ then $p(\lambda)<0$ for all $\lambda\in [0,\mu_{0}[$ with $p(\mu_{0})=0$ proving $ii)$.
\end{proof}

\vspace{0.5 cm}
\noindent
\textbf{Remarks.} 
\begin{enumerate}
\item Recall that in the case $|\Delta(\lambda)|<2$ the equation (\ref{Hequation}) is called \textit{Elliptic}, in such a case all solutions are bounded in the $C^{1}$-norm.  If $|\Delta(\lambda)|>2$ the Hill's equation (\ref{Hequation}) is called \textit{hyperbolic}, in such a case there exists a nontrivial unbounded solution $x_{\lambda}(t)$ of (\ref{Hequation}). Finally, when $|\Delta(\lambda)|=2$ (\textit{Parabolic case}) all solutions are $C^{1}$-bounded if and only if the associated monodromy matrix is $\pm \mathbb{I}_{2}$, with $\mathbb{I}_{2}$ the identity matrix of second  order.

\item Lemma \ref{LH} can be viewed as a generalized and quantified version of classical stability results for parametric Hill's equation with potential $q(t,\lambda)=q(t)+\lambda$ (see \cite{Magnus-Winkler}).
\end{enumerate}
 

\section{Quantifying the bifurcating families from the circular Sitnikov  problem}

Let fix a natural number $N$. The aim of this section is to present the study of the families of even and $2N\pi$ periodic solutions of the Sitnikov problem (\ref{Sit}) parametrized  by the eccentricity  from the quantified point of view, i.e., each family is presented as a graphic of the initial condition as a function of $e$ in a computable interval. This requirement will be essential for the study of the linear stability of this families in our approach. As a by product we will obtain  \textit{ a posteriori} bounds of this families that could be used for the nonlinear stability analysis which is out of the scope of this work. 

\vspace{0.5 cm}
\noindent
Consider the boundary value problem (\ref{bvp}). For given $\xi$, $\eta$ $\in \R$ and $e\in [0,1[$ let $z(t;\xi,e)$ be the solution  of (\ref{Sit}) satisfying the initial conditions
\[
z(0)=\xi, \quad \dot{z}(0)=0.
\]

\noindent
This solution is real analytic in the arguments $(t,\xi,e)\in \R\times \R\times [0,1[$ and is globally defined in $\R$ since the nonlinearity in (\ref{Sit}) is real analytic and bounded. The shooting method allows us to search for even and $2N\pi$-periodic solutions of (\ref{Sit}) by studying the zeros of the function  
\[
F_{N}:\R\times [0,1[\rightarrow \R, \quad F_{N}(\xi,e)=\dot{z}(N\pi;\xi,e).
\]
\noindent
We denote by $\Sigma$ the set of zeros of $F_{N}$, i.e.
\[
\Sigma=\big\{(\xi,e):F_{N}(\xi,e)=0\big\}.
\]
It is a well known fact that $\Sigma$ has  nice topological properties (\cite{Llibre-Ortega, Ortega-Rivera}). First,  for a fixed $e^{*}\in [0,1[$ the section
\[
\Sigma_{e^{*}}=\big\{(\xi,e^{*}):F_{N}(\xi,e^{*})=0\big\},
\] 
is finite (see Proposition 2 in \cite{Ortega-Rivera}). This implies that for each fixed $e\in [0,1[$ there exist a finite number of even sub-harmonics of (\ref{Sit}). Secondly,  $\Sigma$ is bounded (see Proposition 5.1 of \cite{Llibre-Ortega}). More precisely, there exists a positive $\xi_{*}$ such that if $z(t)$ is a even $2N\pi$ periodic solution of (\ref{Sit}) then $|z(t)|<\xi_{*}$ for all $t\in \R.$ 
For instance a numerical computation shows that if $e\in [0, 0.99]$ and $N=1$ then $\xi_{*}=1.99$ (see section 5). Finally every connected  subset of $\Sigma$ is arcwise connected. This corresponds to the intuitive idea of continuation of zeros.


\vspace{0.5 cm}
\noindent
Since $F_{N}$ is odd in $\xi$, the set $\Sigma$ is symmetric with respect to the $\xi$-axis, in consequence it is enough to consider the region 
\[
\Sigma^{+}=\left\{(\xi,e) \in \R^{+}\times [0,1[: F_{N}(\xi,e)=0\right\},
\]
on the right half plane.

\vspace{0.5 cm}
\noindent
In the case $e=0$ (\textsl{The circular Sitnikov problem}) following the results in \cite{Llibre-Ortega} the set $\mathcal{Z}_{0}=\Sigma^{+}\cap\left\{e=0\right\}$ is given by
\[
\mathcal{Z}_{0}=\big\{\xi_{1},\ldots, \xi_{\nu}\big\},
\]
where $\nu:=\nu_{N}=[2\sqrt{2}N]$ and $\xi_{p}$ (with $\xi_{p}>\xi_{q}$ for $q>p$) is the initial condition of the solution $z(t;\xi_{p},0)$  for the Cauchy problem
\begin{equation}\label{csitnikov}
\ddot{z}=-\frac{z}{(z^{2}+1/4)^{3/2}}, \qquad
z(0)=\xi_{p}, \quad \dot{z}(0)=0.
\end{equation}

\noindent
Moreover $\varphi_{p}(t)=z(t;\xi_{p},0)$ has $p$ zeros in $[0,N\pi]$. Therefore, there exists $\nu_{N}$ nontrivial, even and $2N\pi$ periodic solutions in the circular Sitnikov problem  with
\[
\varphi_{1}(0)=\xi_{1}> \varphi_{2}(0)=\xi_{2}>\cdots >\varphi_{\nu}(0)=\xi_{\nu},
\]
labelled according to its number of zeros, going from  $p=1$ to $\nu$. Also in \cite{Llibre-Ortega} the authors prove that Brouwer index of $F_{N}(\xi,0)$ in $\xi_{p}$ denoted by $\text{ind}(F_{N}(\cdot,0),\xi_{p})$ satisfies
\[
\text{ind}(F_{N}(\cdot,0),\xi_{p})=(-1)^{p}.
\]

\noindent
From here it follows that there exists a local branch emanating from $(\xi_{p},0)$ which is the graph of a smooth function $\xi=\mathcal{H}(e)$ with $\xi(0)=\xi_{p}$ for small values of $e$ (See figure 1).

\begin{figure}[h]
\begin{center}
\includegraphics[scale=0.5]{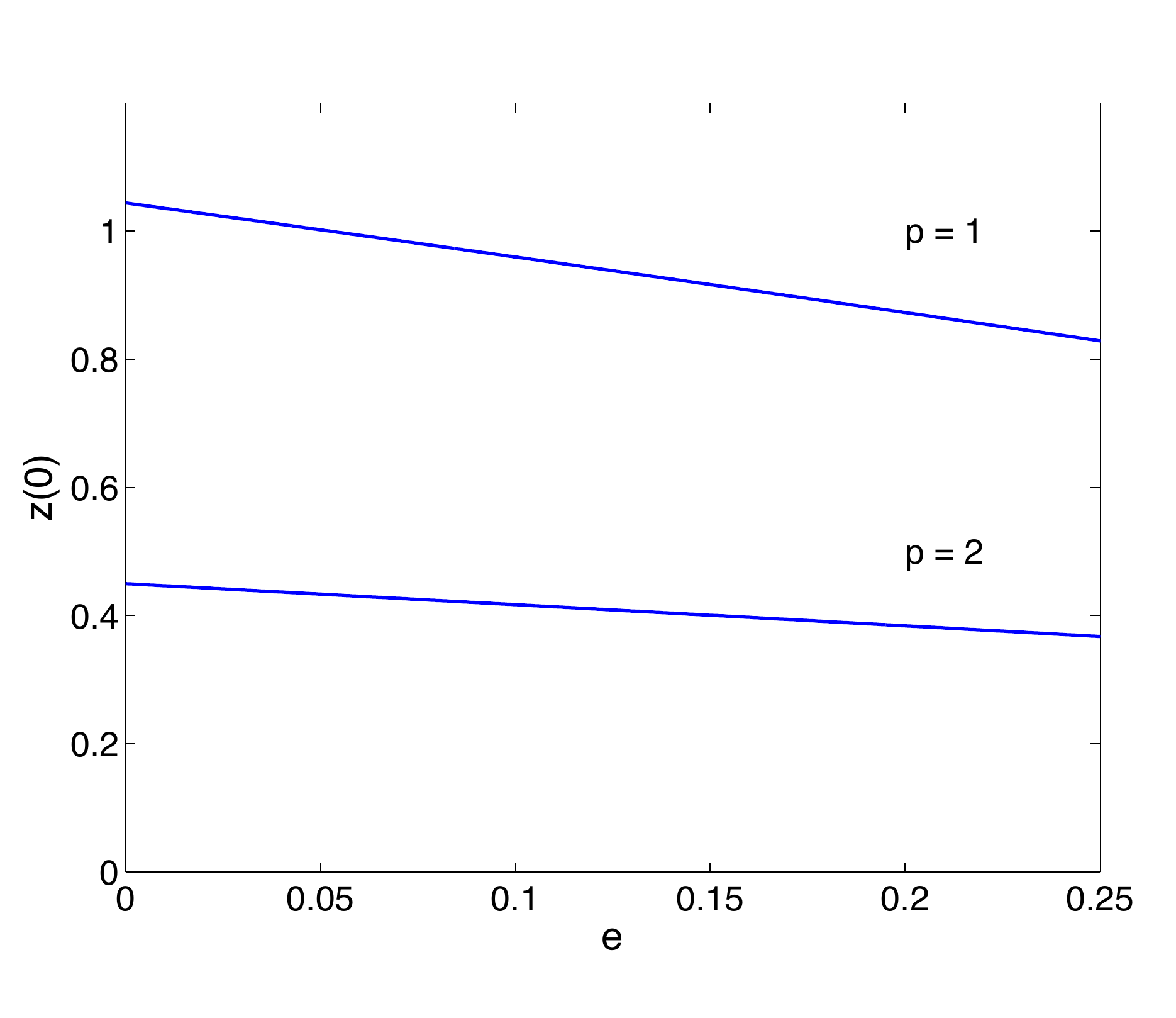}
\caption{Bifurcating families from the circular case for $N=1$ and $p=1$ $p=2$.}
\end{center}
\end{figure}

\vspace{0.5 cm}
\noindent
From now on we fix $0<E<1$ and we shall study the existence of nontrivial solutions of the implicit equation 
\begin{equation}\label{ecuacion implicita}
F_{N}(\xi(e),e)=0, 
\end{equation}
with $(\xi,e)\in [0,\xi_{*}]\times [0,E]$. The equation (\ref{ecuacion implicita}) near to $\xi_{p}$ could be thought (by implicit derivation) as the following Cauchy problem
\begin{equation}\label{ec}
\left\{
\begin{aligned}
\frac{d\xi}{d e}&=h(e,\xi)\\
\xi(0)&=\xi_{p},
\end{aligned}
\right.
\end{equation}
where the function $h(e,\xi)$ is given by
\begin{equation}\label{funcion h}
h(e,\xi)=-\frac{\partial_{e}F_{N}(\xi,e)}{\partial_{\xi}F_{N}(\xi,e)}.
\end{equation} 
Notice that the right hand side in (\ref{ec}) contains the following derivatives of the flow respect to the initial conditions and parameters
\begin{equation*}
\begin{split}
\partial_{e}F_{N}(\xi,e)&=\partial_{e}\dot{z}(N\pi,\xi,e), \\
\partial_{\xi}F_{N}(\xi,e)&=\partial_{\xi}\dot{z}(N\pi,\xi,e).
\end{split}
\end{equation*}
For now on we call the differential equation in (\ref{ec}) the \textit{continuation equation.}

\vspace{0.5 cm}
\noindent
The \textit{continuation equation} makes sense in an open region $\mathcal{U}$ (relative  to $[0,1[\times [0,\infty[ $) where $\partial_{\xi}\dot{z}(N\pi,\xi,e)\neq 0$. So we are interested in a rectangle $\Omega_{p}=[0,E] \times [\xi_{p}-\Delta,\xi_{p}+\Delta]\subset \mathcal{U},$ where $\Delta$ and $E$ will be parameters to be determined. The objectives are to solve (\ref{ec}) in $\Omega_{p}$ starting from  $p=1$ to $\nu$  to obtain:

\begin{itemize}
\item A solution $\xi=\mathcal{H}(e)$ with domain quantified. 
\item Explicit bounds $\left\|Z_{e}\right\|_{\infty}$ for the corresponding even periodic solution $\displaystyle{Z_{e}(t)=z(t;\mathcal{H}(e),e)}$.
\end{itemize}

\noindent
With this in mind, we define $\xi_{0}:=\xi_{*}$ and $\xi_{\nu+1}:=0$ and consider 
\begin{equation}\label{delta}
0<\Delta\leq \Delta_{*}, \quad 0< E\leq E_{*},
\end{equation}
where
\begin{equation}\label{delta-1}
\Delta_{*}=\min\big\{\xi_{p}-\xi_{p+1}: 0\leq p\leq \nu \big\}, \quad E_{*}=0.99.
\end{equation}

This allows us to isolate each initial zero $\xi_{p}$ in each rectangle $\Omega_{p} \subseteq [0,1[ \times [0,\xi_{*}].$

\vspace{0.5 cm}
\noindent
On this approach will lead us to the following main result that will be proved in the subsection 3.3.

\begin{theorem}\label{main theorem} Given a integer $N\geq 1$ and $p=1,\cdots, \cdots [2\sqrt{2}N]$ there exist a constant $e^{*}_{N,p}\in\,]0,1[$, an a affine function $\mathcal{G}:[0,e^{*}_{N,p}]\mapsto \R^{+}$ and a smooth function $\xi=\mathcal{H}_{N,p}(e)$, $e\in [0,e^{*}_{N,p}[$ with $\mathcal{H}_{N,p}(0)=\xi_{p}$ such that $Z^{N}_{e,p}(t)=z(t;\mathcal{H}_{N,p}(e),e)$, is an even $2N\pi$ periodic solution of (\ref{bvp}) with
\[
|Z^{N}_{e,p}|\leq \mathcal{G}(e),
\]
where  $\mathcal{G}(e)=\xi_{p}+\gamma e$ and $\gamma=\gamma_{N,p}$ is a constant that can be  explicitly computed (see (\ref{final bound})).
\end{theorem}

\subsection{Bounds for  the variational equation}
From now on we consider fix an integer $N\geq 1$ and $p=1$ to $\nu=[2\sqrt{2}N]$. We rewrite the equation (\ref{Sit}) in the form 
\begin{equation}\label{nsitnikov}
\ddot{z}+f(t,z,e)=0, \qquad f(t,z,e)=\frac{z}{(z^{2}+r(t,e)^{2})^{3/2}}.
\end{equation}

\noindent
An elementary computations shows that the first variational equation asso\-ciated to (\ref{nsitnikov}) is
\begin{equation}\label{var}
\ddot{y}+a_{\xi,e}(t)y=0, \quad a_{\xi,e}(t):=\frac{r(t,e)^{2}-2z^{2}}{\big(z^2+r^2(t,e)\big)^{5/2}},
\end{equation}
with $z=z(t,\xi,e)$ and $t\in [0,N\pi]$.

\vspace{0.5 cm}
\noindent
In order to find explicit uniform bounds for the numerator and the denominator in (\ref{ec}),  we need to find a uniform bound for the canonical solutions of the variational equation (\ref{var}). For this purpose we will apply the Lemma \ref{LH} to the equation (\ref{var}) with $\xi \in [0,\xi_{*}]$. 

\vspace{0.5 cm}
\noindent
After several computations (see Appendix 1) we have
\begin{equation}\label{ub}
\left|\frac{\partial a_{\xi,e}}{\partial e}\right|\leq \frac{6\sigma\left(1+3N\pi\sigma R_{e}^{2}\right)}{(1-e)^{2}}, 
\end{equation}
for all $t\in [0,N\pi]$, where $R_{e}$ is given as in the Lemma \ref{control hill} taking $\lambda=e$, and 
\begin{equation}\label{sigma}
\sigma=\sigma(e)=\frac{16}{(1-e)^{3}},
\end{equation}
In consequence we can take the function $d(e,R)$ as
\[
d(e,R)=\frac{6\sigma\left(1+3N\pi\sigma R^{2}\right)}{(1-e)^{2}}, 
\]

\noindent
which verifies the assumptions \textit{1.} in Lemma \ref{control hill}. Straightforward computations gives the following expression for the function $Q(e,R)$ 
\[
Q(e,R)=\widetilde{Q}(e,R)+r_{0} \quad \text{where}\quad  \widetilde{Q}(e,R)=b_{1}(e)R^{5}+b_{2}(e)R^{3}-R
\]
with
\[
b_{1}(e)=\frac{9216 (N\pi)^{2}e}{(1-e)^{8}}, \quad b_{2}(e)=\frac{192 N\pi e}{(1-e)^{5}}.
\]
and 
\begin{equation}\label{cota r}
r_{0}=\sup_{\xi\in [0,\xi_{*}]}{R_{0}(\xi)},
\end{equation}
where $R_{0}=R_{0}(\xi)$ is given  by (\ref{R}) for the equation (\ref{var}) with $e=0$. For instance in the case $N=1$, $e\in [0,0.99]$ we have 
\[
\xi_{*}=1.999901.. \quad r_{0}=6.621635..
\]

\noindent
In order to check the assumptions \textit{2.} we present some properties of the function $\widetilde{Q}$.
\begin{itemize}
\item For all $e\in [0,E]$ all roots of $\widetilde{Q}(e,\cdot)$ has two real roots different from zero (one positive and one negative for $e>0$) and they are simple. This follows directly by the positivity of the coefficients $b_{1}(e), b_{2}(e)$.
\item Let $R^{*}(e)$ the first positive root of $\widetilde{Q}$ given by
\[
R^{*}(e)=\sqrt{y^{*}(e)}, \quad y^{*}(e)=\frac{\sqrt{b^{2}_{2}(e)+4b_{1}(e)}-b_{2}(e)}{2b_{1}(e)}
\]
Notice that $\displaystyle{\lim_{e\searrow 0}R^{*}(e)=\infty}.$
\item For $R$ positive, and $e\in [0,E]$ the minimum  value $\widetilde{Q}_{m}(e)$ of $\widetilde{Q}(e,R)$ is given by
\[
\widetilde{Q}_{m}(e)=\widetilde{Q}(e,R_{m}(e))=-2R^{3}_{m}(e)\big(2b_{1}(e)R^{2}_{m}(e)+b_{2}(e)\big),
\]
with $R_{m}(e)$ given by 
\[
R^{2}_{m}(e)=\frac{\sqrt{9b^{2}_{2}(e)+20b_{1}(e)}-3b_{2}(e)}{10b_{1}(e)}.
\]

Notice that 
\[
\lim_{e\to 0}R_{m}(e)=\infty, \quad \lim_{e\to 1}R_{m}(e)=0,
\]
therefore
\[
\lim_{e\to 0}\widetilde{Q}_{m}(e)=-\infty, \quad \lim_{e\to 1}\widetilde{Q}_{m}(e)=0.
\]
Hence, there exits a critical value $E^{*}\in \,]0,1[$ such that $\displaystyle{\widetilde{Q}_{m}(E^{*})=-r_{0}}$
More precisely, $E^{*}$ is the positive root of 
\begin{equation}\label{root}
\widetilde{Q}_{m}(e)=-r_{0}
\end{equation}

This implies $\displaystyle{\widetilde{Q}_{m}(e)<-r_{0}}$ for all $\displaystyle{e\in [0,E^{*}[},$ and therefore there exist exactly two positive roots $\displaystyle{R_{1,e}, R_{2,e}}$ of $\displaystyle{\widetilde{Q}(e,R)=-r_{0}}$ with

\[
\displaystyle{R_{1,e}<R_{m}(e)<R_{2,e}}
\]
verifying 
\begin{equation*}
\begin{split}
Q(e,R)>-r_{0}&\quad (\, \text{i.e.} \, Q(e,R)>0) \quad \forall R \in [0,R_{1,e}[\\
Q(e,R)<-r_{0}&\quad (\, \text{i.e.} \, Q(e,R)<0) \quad \forall R \in \,]R_{1,e},R_{2,e}[
\end{split}
\end{equation*}

\end{itemize}





\noindent
Moreover, for $e\in [E^{*},1[$ it holds $\widetilde{Q}_{m}(e)>-r_{0}$ and therefore $Q(e,R)>0$ for all $R>0$. 

On the other hand,  $r_{0}\leq R_{1,e}$ for all $e\, [0,E^{*}]$. In fact, since $b_{1}(e),b_{2}(e)>0$ for $e>0$ on has 
\[
Q(e,R)>-R+r_{0}, \quad \forall e\in ]0,1[,\quad  R>0,
\]
then
\[
0=Q(e,R_{1,e})>-R_{1,e}+r_{0} \quad \Rightarrow \quad R_{1,e}>r_{0}.
\]
Finally the assumption 2. (part (b)) is fulfilled for all $e\in [0,E^{*}[$ and in consequence 
\begin{equation}\label{canonic bound}
R_{e}\leq R_{m}(E^{*}), \quad \forall e \in \,[0,E^{*}]
\end{equation}

\noindent
From now on, in the rest of the paper we use the following notation
\[
\mathcal{R}=R_{m}(E^{*}).
\]

\begin{figure}[h]
\begin{center}
\includegraphics[scale=0.5]{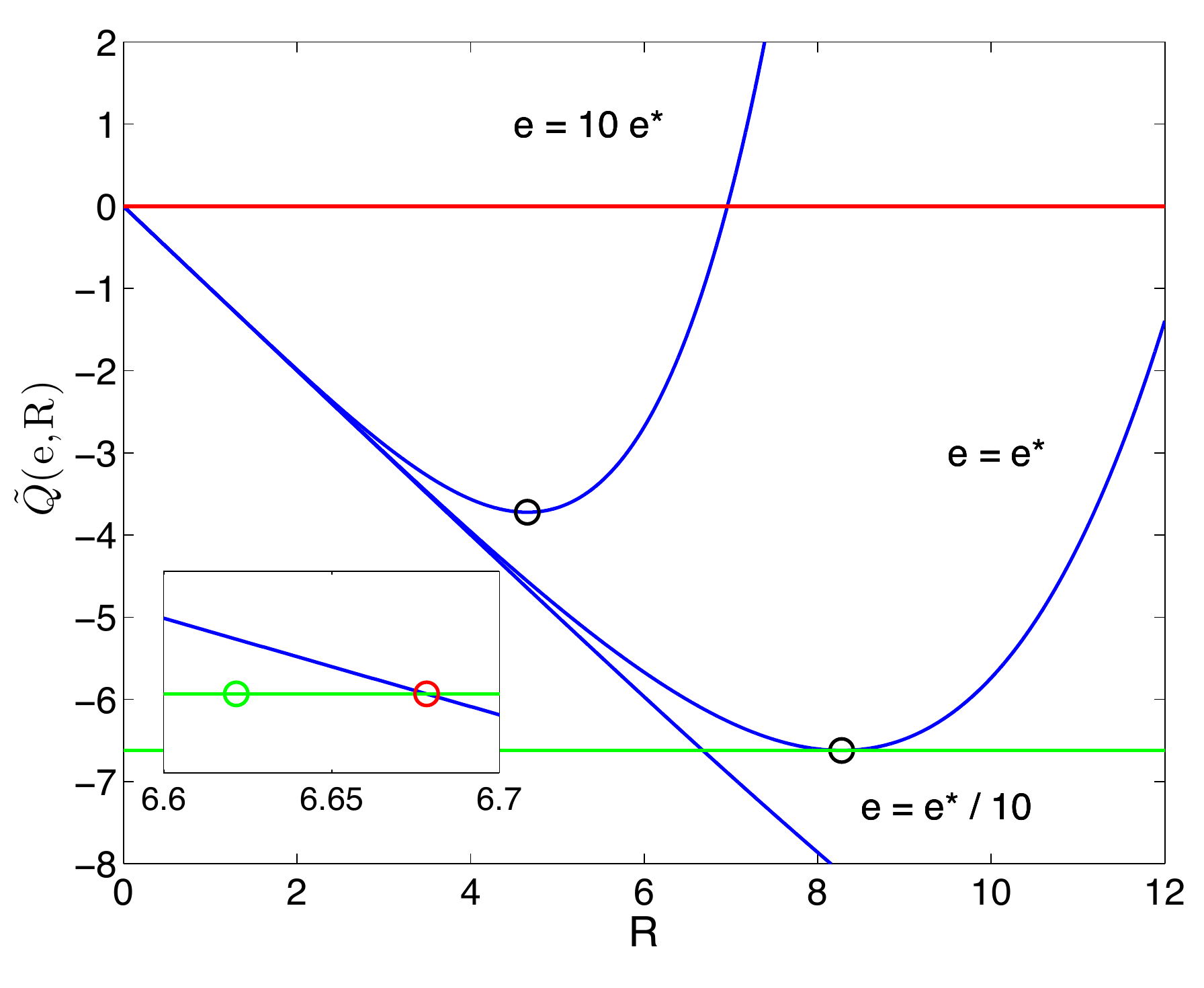}
\caption{Function $\widetilde{Q}(e,R)$ and the  behaviour of the solutions of the equation $\widetilde{Q}(e,R)=-r_{0},$  i.e., $Q(e,R)=0$.}
\end{center}
\end{figure}

\subsection{Bounds for the continuation equation}

In this part of the manuscript we present an explicit positive upper bound for $\partial_{\xi}F_{N}(\xi,e)$, the numerator on the continuation equation  (\ref{ec}) and also a positive lower bound for  $\partial_{\xi}F_{N}(\xi,e)$, the respective denominator. Remember that we have fixed $N\geq 1$ and $\displaystyle{p=1}$ to $\nu$. 

\vspace{0.5 cm}
\noindent
We star with the numerator, with this in mind, notes that $\beta(t)=\displaystyle{\partial_{e}z(t,\xi,e)}$ solves the initial value problem
\begin{equation}\label{var2}
\begin{split}
&\ddot{y}+a_{\xi,e}(t)y=p(t,\xi,e),\\
&y(0)=\dot{y}(0)=0.
\end{split}
\end{equation}
with $\displaystyle{p(t,\xi,e)=-\partial_{e}f(t,z,e)}$, $z=z(t,\xi,e)$ and $\xi\in ]\xi_{p}-\Delta^{*},\xi_{p}+\Delta^{*}[$. Using (\ref{kepler}) the computations show that
\begin{equation}\label{funcion p}
p(t,\xi,e)=\frac{3rz}{2(z^{2}+r^{2})^{5/2}}\Big(\cos u-e\frac{\sin^{2}u}{2r}\Big).
\end{equation}

\noindent
From (\ref{kepler}) the function $r(t,e)$ satisfies 
\begin{equation}\label{cota kepler}
\frac{1-e}{2}\leq r(t,e)\leq \frac{1+e}{2},
\end{equation}
for all $(t,e)\in \R\times [0,1[$, therefore we get
\begin{equation*}\label{norma de p}
|p(t,\xi,e)|\leq \frac{3}{2(z^{2}+r^{2})^{3/2}}\Big(1+\frac{e}{2r}\Big),
\end{equation*}
in consequence
\begin{equation}\label{norma de p}
\left\|p\right\|_{\infty}\leq\frac{12}{(1-E)^{4}}.
\end{equation}

\vspace{0.5 cm}
\noindent
From the method of variations of parameters, we obtain
\begin{equation}\label{ecuacion integral}
\beta(t)=\int_{0}^{t}G(t,s,e)p(s,\xi,e)ds,
\end{equation}
where $\displaystyle{G}(t,s,e)=\phi_{1}(s,e)\phi_{2}(t,e)-\phi_{1}(t,e)\phi_{2}(s,e)$ and $\phi_{1}$ and $\phi_{2}$ are the solutions of (\ref{var}) satisfying the initial conditions
\[
y(0)=1,\quad  \dot{y}(0)=0 \quad \text{and} \quad y(0)=0, \quad  \dot{y}(0)=1,
\]
respectively. With the previous results we are able to present a uniform bound on $\Omega_{p}$ for the function $\dot{\beta}(t)$ on $[0,N\pi].$ Integrating over the equation (\ref{var2}) we arrive to
\[
\dot{\beta}(t)=-\int_{0}^{t}a_{\xi,e}(s)\beta(s)ds+\int_{0}^{t}p(s,\xi,e)ds.
\]
From (\ref{ub}) and (\ref{sigma}) we get
\begin{equation*}
|\dot{\beta}(N\pi)|\leq \frac{12(N\pi)}{(1-e)^{4}}\Big(1+N\pi\sigma \left\|G(\cdot,\cdot,e)\right\|_{\infty}\Big),
\end{equation*}
valid  for all $(e,\xi)\in \Omega_{p}=[0,E^{*}]\times [\xi_{p}-\Delta_{*},\xi_{p}+\Delta_{*}]$. Here $\left\|G(\cdot,\cdot,e)\right\|_{\infty}$ is defined as
\[
\left\|G(\cdot,\cdot,e)\right\|_{\infty}=\sup_{[0,N\pi]^{2}}|G(\cdot,\cdot,e)|.
\]
From (\ref{canonic bound}) it follows $\left\|G(\cdot,\cdot,e)\right\|_{\infty}\leq 2\mathcal{R}^{2}$, in consequence we obtain a uniform bound over the numerator of the continuation equation (\ref{ec}) on $\Omega_{p}$ as follows

\begin{equation}\label{cota numerador}
\big|\partial_{e}F_{N}(\xi,e)\big|\leq \frac{12 N\pi}{(1-E^{*})^{4}}\Big(1+2N\pi\sigma^{*} \mathcal{R}^{2}\Big):=\Upsilon.
\end{equation}
with $\sigma^{*}=\sigma(E^{*})$  given by (\ref{sigma}).

\vspace{0.5 cm}
\noindent
For the denominator, we proceed as before.  Notice that $\phi_{1}(\cdot,e)=\partial_{\xi}z(\cdot,\xi,e)$, is the canonical solution of (\ref{var}) that satisfies the initial conditions
\[
y(0)=1, \quad \dot{y}(0)=0.
\]
By the Mean Value Theorem we obtain 
\[
\big|\dot{\phi_{1}}(t,e)-\dot{\phi}_{1}(t,0)\big| \leq \max_{e \in [0,E^{*}]} \|\partial_{e}\dot{\phi}_{1}(\cdot,\,e)\|_{\infty}\,e,
\]
\noindent
for all $t\in [0,N\pi]$. The function $\partial_{e}\phi_{1}(t,e)$ satisfies the Cauchy problem \begin{equation}\label{cp}
\left\{
\begin{aligned}
\ddot{y}+a_{\xi,e}(t)y&=-\partial_{e}a_{\xi,e}(t)\phi_{1}(t,e)\\
y(0)=\dot{y}(0)&=0.
\end{aligned}
\right.
\end{equation}

\noindent
In consequence
\begin{equation}
\partial_{e}\dot{\phi}_{1}(t,e)=-\int_{0}^{t}\partial_{t}G(t,s,e)\partial_{e}a_{\xi,e}(s,e)\phi_{1}(s,e)ds,
\end{equation}
therefore, using (\ref{ub}) and (\ref{canonic bound}) we arrive at
\begin{equation}\label{cotadervarphi1}
\big|\partial_{e}\dot{\phi}_{1}(t,e)\big|\leq \frac{12 N\pi \sigma^{*}(1+3N\pi\sigma^{*}\mathcal{R}^{2})\mathcal{R}^{3}}{(1-E^{*})^{2}}:=\Psi
\end{equation}
Finally, 
\begin{equation}\label{cota phi}
\big|\dot{\phi_{1}}(N\pi,e)-\dot{\phi}_{1}(N\pi,0)\big|\leq  e\Psi
\end{equation}
\noindent
Now we impose the restriction 
\begin{equation*}\label{restriccion 2}
(\textbf{R}) \quad e\Psi < \big|\dot{\phi}_{1}(N\pi,0)\big|
\end{equation*}

\noindent
Combining (\ref{cota phi}) and (\textbf{R}) we obtain \footnote{For all $a,b, c \in \R$ if $|a-b|<c<|b|$ then $|a|>|b|-c$.}
\begin{equation}\label{cota denominador}
|\dot{\phi}_{1}(N\pi,e)|>\Gamma(e):=|\dot{\phi}_{1}(N\pi,0)|-e\Psi>0,
\end{equation}
for all $e\in [0,E^{**}[$ where $E^{**}$ is given by
\begin{equation}
E^{**}=\min\left\{E^{*},\hat{E} \right\}, \quad \hat{E}=\frac{|\dot{\phi}_{1}(N\pi,0)|}{2\Psi}.
\end{equation}

\noindent
Therefore the Cauchy problem (\ref{ec}) is well defined  in the rectangle
\[
\Omega^{*}_{p}=[0,E^{**}[\times ]\xi_{p}-\Delta_{*},\xi_{p}+\Delta_{*}[,
\]
with $\Delta_{*}$ given by (\ref{delta-1}), and moreover from (\ref{cota numerador}) and (\ref{cota denominador})  we get (see (\ref{funcion h}))
\begin{equation}\label{1 cota M}
\left|h(e,\xi)\right|\leq M=\Upsilon/\Gamma.
\end{equation}
with $\Gamma=|\dot{\phi}_{1}(N\pi,0)|-\Psi E^{**}$ and $M$ given by
\begin{equation}\label{2cota M}
M=\frac{12N\pi \big(1+2N\pi\sigma^{*} \mathcal{R}^{2}\big)}{(1-E^{*})^{4}\Big(\big|\dot{\phi}_{1}(N\pi,0)\big|-\Psi E^{**}\Big)}.
\end{equation}




\subsection{Proof of Theorem \ref{main theorem}}

Following the previous results in the subsections 3.1 and 3.2, we are able to proof the Theorem \ref{main theorem} stated in the section 3. For a fixed integer $N\geq 1$ and $p=1$ to $\nu$  we can apply the existence Peano's Theorem in $\Omega^{*}_{p}$ to conclude that there exists a solution $\xi=\mathcal{H}(e)$, $e\in [0,e^{*}[$ of the \textit{continuation equation} (\ref{ec}) with
\begin{equation}\label{cota e}
e^{*}=e^{*}_{N,p}:=\min\left\{E^{**},\frac{\Delta_{*}}{M}\right\}.
\end{equation}

\noindent
In consequence, we obtain an nontrivial, even, $2N\pi$-periodic solution $Z^{N}_{e,p}(t):=z(t;\mathcal{H}_{N,p}(e),e)$ of (\ref{Sit}) as a continuation of the solution $\varphi_{p}(t)=z(t;\xi_{p},0)$ for $e\in [0,e^{*}[$. On the other hand, 

\begin{equation}
\begin{split}
\left|\frac{\partial Z^{N}_{e,p}}{\partial e}\right|&\leq \left|\frac{\partial z}{\partial \xi}\right|\,\left|\frac{d\xi}{d\,e}\right|+\left|\frac{\partial z}{\partial e}\right|\\
&\leq M\mathcal{R}+\frac{3N\pi\sigma^{*} \mathcal{R}^{2}}{2(1-e^{*})}.
\end{split}
\end{equation}

\noindent
The  estimative for $\displaystyle{\left|\frac{\partial z}{\partial e}\right|}$ can be found in the Appendix 1. Finally, by the Mean Value Theorem we arrive to
\begin{equation}\label{final bound}
|Z^{N}_{e,p}|\leq \xi_{p}+ \gamma e :=\mathcal{G}(e),
\end{equation}
with $\gamma := \gamma_{N,p}$ given by
\[
\gamma :=\left( M\mathcal{R}+\frac{3N\pi\sigma^{*} \mathcal{R}^{2}}{2(1-e^{*})}\right).
\] 
and this complete the proof. $\square$

\section{Linear Stability}
\bigskip
In the previous sections we have found, for each natural number $N$ and $1\leq p \leq \nu \, ,\,\,\,\nu=[2\sqrt{2}N],$ an even $2N\pi-$periodic family of solutions $Z_{e,p}^{N}$ of the Sitnikov problem, bifurcating from the circular ones $\varphi_p=Z_{0,p}^{N}$ (see section 2) and with the remarkable property of having $p$ zeros on $[0, N\pi],$ see Lemma 7.2  in \cite{Llibre-Ortega}. This family was parametrized by the eccentricity $e$ in a computable interval. Now, we will search the stability properties of the families  $Z_{e,p}^{N}$  at least in the linear sense. For this purpose we deal with \textit{the discriminant function} associated with the first variational equation along to the periodic solution $Z_{e,p}^{N},$   

\begin{equation}\label{varper}
\ddot{y}+q(t,e,p,N)y=0,
\end{equation}
where 
\begin{equation}\label{avar}
q(t,e, p, N):=\frac{r(t,e)^2-2Z_{e,p}^{N}(t)^2}{(Z_{e,p}^{N}(t)^2+r(t,e)^2)^{5/2}}.
\end{equation}

\noindent
Hereinafter we fix $N$-odd and we denote $Z_e (t):=Z_{e,p}^{N}(t)$, and $q(t,e):= q(t,e,p,N)$, in order to simplify the notations. For $i=1,2,$ let $y_i (t,e)$ be the canonical solutions of (\ref{varper}),  satisfying  
$$
y_1(0,e)=\dot{y_2}(0,e)=1,\qquad \dot{y_1}(0,e)=y_2(0,e)=0.
$$
The {\it discriminant\, function}  associated to  (\ref{varper}) is defined by

\begin{equation}\label{discri}
\Delta (e)=y_1(2N\pi, e)+ \dot{y_2}(2N\pi, e),
\end{equation}
which is the trace  of the monodromy matrix associated to the first order planar system regard to the Hill's equation (\ref{varper}). It is a well known fact that (\ref{varper}) is stable (equivalently $Z_{e}$ is \textit{linearly stable}) if and only if the corresponding Floquet's multipliers $\rho_{1}(e)$, $\rho_{2}(e)$ satisfy some of the following conditions:  

\begin{enumerate}
  \item [i)] $\rho_1(e)=\overline{\rho_2}(e)\notin \R,\,\, |\rho_{1,2}(e)|=1$ ( \textit{Elliptic Case} ),
  \item [ii)] $\rho_{1,2}(e)=\pm 1$ and the monodromy matrix is equal to $\pm I_{d}$ where $I_{d}$ is the identity matrix, i.e. $\dot{y_1}(2N\pi,e)=y_2(2N\pi,e)=0$  (\textit{Stable Parabolic Case}).
	\end{enumerate}

\noindent
Notice that the Elliptic case is equivalent to have $|\Delta (e)|<2,$ and in the Stable Parabolic case one has $|\Delta(e)|=2$. In particular we have $\Delta (0)=2$, since the function $\displaystyle{\dot{\varphi}_{p}}$ is a $2N\pi$-periodic solution of (\ref{varper}) with $e=0$ (how a direct computation shows) and therefore $\rho_{1,2}(0)=1$ in the circular case.

\vspace{0.5 cm}
\noindent
Following a standard approach as in \cite{Magnus-Winkler} the formula of $\Delta^{\prime}(e)$ is given by
\[
\Delta^{\prime}(e)=-\left[\int_{0}^{2N\pi}\Big(G(2N\pi,s,e)y_{1}(s,e)+\partial_{t}G(2N\pi,s,e)y_{2}(s,e)\Big)\partial_{e}q(s,e)ds\right]
\]
where $\displaystyle{G(t,s,e)=y_{1}(s,e)y_{2}(t,e)-y_{1}(t,e)y_{2}(s,e)}$. Since $q(t,e)=q(-t,e)$ for all $(t,e)\in \R\times e \in [0,1[$ from the Theorem 1.1 in \cite{Magnus-Winkler} we obtain
\begin{equation}\label{derivada discri-e}
\Delta^{\prime}(e)=-\left[\int_{0}^{2N\pi}\Big(y_{1}^{2}(s,e)y_{2}(2N\pi,e)-\dot{y}_{1}(2N\pi,e)y_{2}^{2}(s,e)\Big)\partial_{e}q(s,e)ds\right]
\end{equation}

\noindent
In particular for $e=0$ we have
\begin{equation}\label{derivada discri-e cero}
\Delta^{\prime}(0)=\dot{y}_{1}(2N\pi,0)\int_{0}^{2N\pi}y_{2}^{2}(s,0)\partial_{e}q(s,0)ds,
\end{equation}

\noindent
since $y_{2}(t,0)$ is a multiple of $\displaystyle{\dot{\varphi}_{p}}$ and therefore is odd and $2N\pi$-periodic,  in consequence $y_{2}(2N\pi,0)=0.$ Moreover, by the Theorem 1.1 and 1.2 in \cite{Magnus-Winkler} we deduce that $\displaystyle{\dot{y}_{1}(2N\pi,0)\neq 0}$ therefore $Z_{0}$ is \textit{linearly unstable}.




\vspace{0.5 cm}
\noindent
The study of the sign of $\Delta '(0)$ clearly implies  a stability result for the linea\-rized equation (\ref{varper}) for small $e$. However, some numerical computations reveals that this quantity could be nule. This fact is not deducible from (\ref{derivada discri-e cero}), and in order to prove it we shall  consider  ``negative eccentricities'' in the Sitnikov equation. The first observation (see the Appendix 2) is the following

\noindent
{\it The function $r(t,e)$ can be analytically extended for $e\in [-0.6627434... ,0], $ and it verifies for $N$ odd}

\begin{equation}\label{-e}
r(t,-e)=r(t+N\pi, e),\qquad \forall t\in \R,\,\forall e\in [0,0.6627434...]. 
\end{equation}
This implies that the extended Sitnikov equation (\ref{Sit}) will be analytical for small $|e|$. Thus, we can consider again  the continuation equation (\ref{ec}) which makes sense and is analytical in $e$ on a small $\xi e-$rectangle centered in $(\xi_p,\,0)$ for each $p=1$ to $\nu.$
A similar procedure like in the section 3 led us to the existence of an unique function $\xi=h_{p}(e)$ for small $|e|$ such that $h_{p}(0)=\xi_p>0$ and the solution $ Z_{\pm e}(t)=z(t,h_p(\pm e),\pm e),$ of the extended Sitnikov equation
\[
\ddot{z}+\frac{z}{(z^{2}+r(t,\pm e))^{3/2}}=0,
\]
is even and $2N\pi$-periodic for $0\leq e< 0,6627434...$ 

\begin{lemma} For $N$ odd we have
\begin{equation*}\label{relacion}
Z_{-e}(t)=(-1)^p Z_e (t+N\pi).
\end{equation*}
\end{lemma}

\begin{proof}
\noindent
In fact, $\xi=h_p(-e)$ is by definition the unique initial condition such that $$\displaystyle{F_N(\xi,-e)=\dot{z}(N\pi, h_p(-e),0,-e)=0}.$$
On the other hand, because the relation (\ref{-e}), the Sitnikov equation for $-e$ can be written like 

\begin{equation}\label{Sitnikov-e} 
\ddot{z}=-f(t,z,-e)=-f(t+N\pi,z,e).
\end{equation}
From here is clear that if $z(t)$ is a solution for $e>0$ then $z(t+N\pi)$ is a solution for $-e$. This implies the following identity in term of flows

\begin{equation}\label{relacionflows}
  z(t,z(N\pi,\xi,0,e),0,-e)=z(t+N\pi,\xi,e),
	\end{equation}
where $z(t,\xi,\eta,e)$ denotes the general solution for the extended Sitnikov equation. With this in mind it is not difficult to prove the following interesting relation
\begin{equation}\label{rel2}
h_p(-e)=(-1)^p z(N\pi,h_p(e),e).
\end{equation}
From the symmetry $f(t,(-1)^{p}z,e)=(-1)^{p}f(t,z,e)$, we have that $y(t)=(-1)^{p}z(t+N\pi,h_p(e),e)$ is a solution of (\ref{Sitnikov-e}) and satisfies
\[
y(0)=(-1)^{p}z(N\pi,h_p(e),e), \quad \dot{y}(0)=(-1)^{p}\dot{z}(N\pi,h_p(e),e)=0.
\]

\noindent
Thus $y(0)$ will be $h_p(-e)$ depending on the sign of $z(N\pi,h_p(e),e),$ since $h_p(-e)>0$ for small $|e|$. On the other hand 
\[
z(N\pi,h_p(e),e)h_p(e)\gtrless 0, \quad \text{for $p$ even(odd)},
\]
because $Z_e(t)$ has $p$ zeroes at $]0,N\pi[. $ If $p$ is even necessarily $h_p(-e)=z(N\pi,h_p(e),e),$ and if $p$ is odd necessarily $h_p(-e)=-z(N\pi,h_p(e),e).$ This prove (\ref{rel2}). Finally the proof of (\ref{relacion}) is as follows. Using (\ref{rel2}), (\ref{relacionflows}) and the symmetries of $f$  we arrive at
\begin{equation}
\begin{split}
Z_{-e}(t)=z(t,h_p(-e),-e)&=z(t,(-1)^p z(N\pi,h_p(e),e),-e)\\
&=(-1)^{p} z(t,z(N\pi,h_p(e),e),-e)\\
&=(-1)^{p}z(t+N\pi,h_p(e),e)\\
&=(-1)^{p}Z_e(t+N\pi)
\end{split}
\end{equation}
and this complete the proof.
\end{proof}

\begin{Prop}\label{paridad}
The discriminant function $\Delta(e)$ given by (\ref{derivada discri-e}) is well defined for small $|e|$ and moreover is an even function.
\end{Prop}

\begin{proof}
Let us consider the first order linear periodic system associate to (\ref{varper})
\begin{equation}\label{linear system}
\begin{split}
\dot{u}&=v\\
\dot{v}&=-q(t,e)u
\end{split}
\end{equation}
Let $\Phi(N\pi;t)$ be the fundamental matrix of (\ref{linear system})  which is principal at $t=N\pi$. Since $q(t,-e)=q(t+N\pi,e)$ then $\Psi(0;t)=\Phi(N\pi;t+N\pi)$ is a fundamental matrix of the system
\begin{equation}\label{linear system-2}
\begin{split}
\dot{u}&=v\\
\dot{v}&=-q(t,-e)u
\end{split}
\end{equation}
which is principal at $t=0.$ Henceforth the systems (\ref{linear system}) and (\ref{linear system-2}) share the same monodromy matrix $\Phi(N\pi;3N\pi)$. In consequence
\[
\Delta(e)=\Delta(-e).
\] 
This completes the proof.
 \end{proof}

\vspace{0.5 cm}
\noindent
Since $\Delta(0)=2$ and $\Delta^{\prime}(0)=0$, the study of the stability for the family $Z^{N}_{e,p}$ depends on the sign of $\Delta^{\prime \prime}(0)$. For instance, if $\Delta^{\prime \prime}(0)>0$, then $\Delta(e)>2$ for small positive $e$ and in consequence $Z^{N}_{e,p}$ is unstable. Elsewhere if $\Delta^{\prime \prime}(0)<0$ we have that $Z_{e}$ is stable for small positive $e$. How small? To answer this question we will apply the Lemma \ref{LH}, therefore will be necessary to compute some constants in that lemma. First we initiate with an estimation of $|\Delta^{\prime \prime \prime}(e)|$ on $[0,e^{*}]$ for the cases $N=1$ and $N=3$. From the (\ref{derivada discri-e}) we have 

\begin{equation}
\Delta^{\prime}(e)=-\left[\int_{0}^{2N\pi}w(t,s,e)\partial_{e}q(s,e)ds\right]
\end{equation}
with $\displaystyle{w(s,e)=y_{1}^{2}(s,e)y_{2}(2N\pi,e)-\dot{y}_{1}(2N\pi,e)y_{2}^{2}(s,e)}.$ Therefore 
\begin{equation}\label{derivada segunda-discri}
\Delta^{\prime \prime}(e)=-\left[\int_{0}^{2N\pi}\Big(\partial_{e}w(s,e)\partial_{e}q(s,e)
+w(s,e)\partial_{e}^{2}q(s,e)\Big)ds\right]
\end{equation}
and
\begin{equation}\label{derivada tercera-discri}
\Delta^{\prime \prime \prime}(e)=-\left[\int_{0}^{2N\pi}\Big(\partial_{e}^{2}w(s,e)\partial_{e}q(s,e)
+2\partial_{e}w(s,e)\partial_{e}^{2}q(s,e)+w(s,e)\partial_{e}^{3}q(s,e)\Big)ds\right]
\end{equation}
\vspace{0.7 cm}
\noindent

\begin{theorem}
Let $N\in \N $ odd, and $p=1$ to $\nu$ fixed.  For $e\in [0,e^{*}]$ where $e^{*}=e^{*}_{N,p}$ given by the Theorem \ref{main theorem}, let $\Delta(e)$ the discriminant function defined by (\ref{discri}), $\displaystyle{\mathcal{K}=\sup_{e\in[0,e^{*}]}|\Delta^{\prime \prime \prime}(e)|}$ and $\mu$, $\mu_{0}$ and $p(e)$  defined as in Lemma \ref{LH} with $\lambda=e$ for the Hill's equation (\ref{varper})-(\ref{avar}). Then, the periodic solution $Z^{N}_{e,p}$ is:
\begin{enumerate}
\item Hyperbolic if $\Delta^{\prime \prime}(0)>0$ for all $e\in \,]0,\min\left\{\mu,e^{*}\right\}[$.
\item Elliptic if $\Delta^{\prime \prime}(0)<0$  and $p(e^{*})>0$ (resp. $p(e^{*})\leq 0$) then $|\Delta(e)|<2$ for all $\lambda \in I_{2}=]0,\min\left\{\mu,\mu_{0},e^{*}\right\}[$ (resp. $\lambda \in I_{1}$) where $\displaystyle{\mu_{0}}$ is the unique positive root of $\displaystyle{p(e)=0.}$
\end{enumerate}
\end{theorem}

\begin{proof}
The proof follows as a direct consequence of the Lemma \ref{LH}, the Proposition \ref{paridad} applied to the Hill's equation (\ref{varper})-(\ref{avar}).
\end{proof}

\section{Numerical Results}

In this section we discuss the application of the previous theoretical results to the
even and periodic solutions of the Sitnikov problem with $N=1$ and $N=3$ to obtain
the numerical values of the quantified interval of existence of the branches 
and their stability.

The detailed calculations for other values of $N$,
 the numerical results for values of eccentricity
close to one, the countable number of  branches that emanate from the
 trivial equilibrium solution as well as  the comparison with 
 previous results  \cite{Jimenez,Martinez-Chiralt} will be presented elsewhere.

The different variables and coefficients that have to be evaluated for the quantification
of the intervals depend on the $2N\pi$-periodic solutions of the integrable circular problem ($e=0$) and in some cases on the solution along the bifurcating branch of periodic orbits for the non-integrable case ($e\neq0$).

The first quantities can be easily computed by direct integration of the full and linearized equations once the initial condition $z(0)$ that correspond to even periodic solutions has been found.

However, the non zero eccentricity quantities require the explicit calculation of the emanating branch  ($\xi=\mathcal{H}(e)$) which can be obtained by 
numerical continuation. We have made used of the continuation procedure presented in \cite{Munoz1} for the conservative case and later extended to properly treat 
the symmetries and reversibilities in \cite{Munoz1}. See also \cite{Galan} for a review and examples from Mechanics. 

In the Sitnikov problem a two steps procedure has been necessary; first we have
continued the circular family of period orbits for $e=0$ parametrized by the period.
We have detected the initial conditions $\xi$  whose associated period is commensurate with that of the primaries. Precisely with that initial condition we have computed by 
initial value integration an appropriate starting solution for the emanating branch that was the input of a boundary value continuation in the eccentricity. 
The result is the branch that can be labelled by $N$ and $p$ where $p$ is the number of zeros in half a period $[0,N\pi]$. As a by product of the numerical continuation we compute with negligible cost the multipliers of the $2N\pi$ periodic solution and detect the possible
bifurcations.

The final outcome of the calculation is a branch in the $\xi,e$ plane for each 
$N$ and $p$ and  the linear stability of the associated periodic solution. In Figure 1 we plot the two branches for $N=1$ in a reduced interval of eccentricities ( $[0,0.25]$). The numerical results shows that the branches extend up to eccentricities close to 1 with a change of stability along the way.


It is a straightforward calculation to evaluate the different quantities
that are needed in our quantitative stability analysis.
The quantities that, in principle, do not depend on $p$ for a fixed $N$  are cast in table 1
for $N=1$ and $N=3$: $\xi_*$, $r_0$, $\mathcal{R}$ and $E^*$. However, for the two cases considered the value $r_0$ is the same and consequently also for $\mathcal{R}$ and $E^*$.

\begin{table}[h]
\begin{center}
\begin{scriptsize}
\begin{tabular}{|c|c|c|c|c|}
\hline
 N &$\xi_*$ &  $r_0$ & $\mathcal{R}$ & $E^*$   \\ \hline \hline
1  &  1.999901 & 6.621636   &	8.277124 &  4.684299 e-10	 \\ \hline	
3  &   4.160101 &  6.621636   & 8.277124 & 4.684299 e-10		\\
\hline
\end{tabular}
\end{scriptsize}
\end{center}
\caption{Numerical results for the relevant variables in the quantification
for $N=1$ and $N=3$.
\label{tbl:quantification}}
\end{table}

For a fixed $N$, the \textit{a priori} bound $\xi_*$  for the initial conditions of the $2N\pi$-periodic solutions of (\ref{Sit}) is computed by comparison with an auxiliary circular Sitnikov problem with an appropriate radius. $\xi_{*}$ is thee initial displacement corresponding to a period $4 N \pi$  (see \cite{Llibre-Ortega}). It can be computed from the
analytical expression of the period function.

The estimation of the upper bound $r_0$ for the canonical solutions deserves a comment; it does not depend on the branch and has to be valid for the 
whole $[0, \xi_*]$ interval of initial conditions. We have computed $R_0(\xi)$ (equation (\ref{cota r}))
with $0\leq \xi \leq \xi_*$ (see figure \ref{fig:r0}) and its supreme value $r_0$.
\begin{figure}[h]
\begin{center}
\includegraphics[scale=0.5]{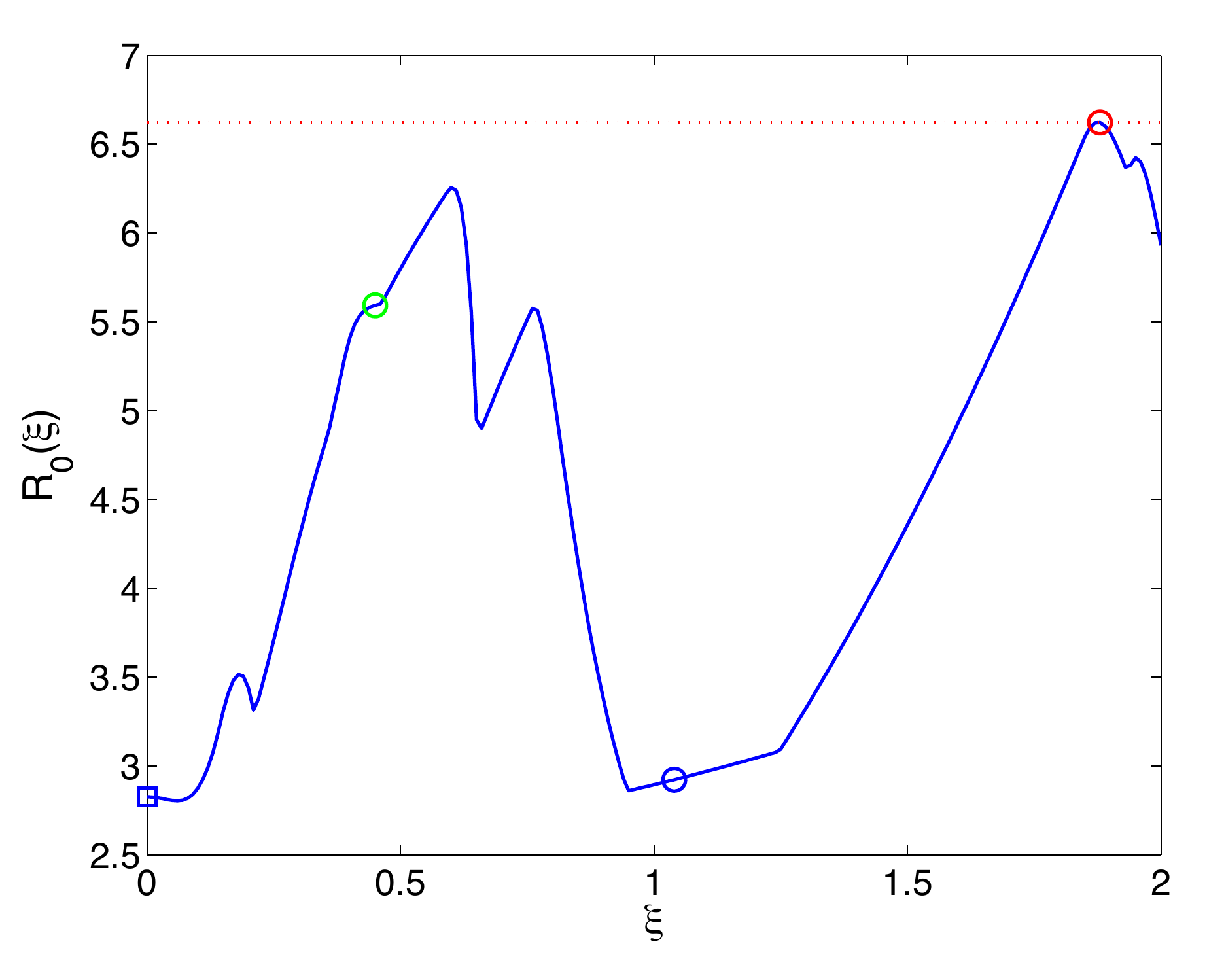}
\caption{Estimation of the upper bound $r_0$ for $N=1$.
 The function  $R_0(\xi)$  is plotted for $0\leq \xi \leq \xi_*=1.999901$.
The leftmost blue square coincides with the analytical value for  $e=0$
at a height of $2\sqrt{2}$. The green (blue) circle corresponds to the $p=2$ ($p=1$)
starting value of the branches. The red circle and the dotted red line 
indicates the upper bound valid in the whole range of values of $\xi$. 
 \label{fig:r0} }
\end{center}
\end{figure}

\vspace{0.5 cm}
\noindent
In tables 2 and 3 we list the coefficients that do depend on the specific branch $p$ for N$=1$ and $N=3$ respectively: 
 $\hat{E}$, $\Delta''(0)$, 
$\mathcal{K}$ and $\mu_0$.

In all the cases analysed  in this work the sign of $\Delta^{\prime \prime}(0)$ turns out to be negative; i.e., all the families of periodic solutions are elliptic for all $e\in [0,E^{*}]$, because the rest of bounds for the 
existence and stability of the families are less restrictive than $E^*$.

The final result for the interval length is a small quantity ($\sim 10^{-10}$)
specially if compared with the numerical continuation result that extends it up to $e\sim 0.6$
along some of the branches. We should highlight that in our case $E^{*}$ is a 
rigorously quantified value and almost 30 orders of magnitude larger than 
the value of a standard quantification via the application of the fundamental inequality comparing with the circular problem
that generates exponentially small intervals. Here the key ingredient has been the use of Lemma \ref{control hill}. The moderate change in the value of $\mathcal{R}$ compared with the
starting value of $r_0$ indicates that our novel quantification Lemma \ref{control hill} is an useful tool for the quantification of the canonical function  and their derivatives.

Besides, a higher order bound for some 
of the expressions that appear in the quantification would significantly increase the
interval of validity but would introduce more complexity and technical details to the
analysis. For the sake of simplicity we have decided to use only first order estimates. 

\begin{table}
\begin{center}
\begin{scriptsize}
\begin{tabular}{|c|c|c|c|c|}
\hline
 p & $\hat{E}$ & $\Delta''(0)$ &   $\mathcal{K}$ & $\mu_0$ \\ \hline \hline
1  &  6.2314169e-10  & -10.10096	 & 1	& 0.88995 \\ \hline	
2  &  1.582592e-9   &  -0.034051(*) & 1 &	15.328 	\\
\hline
\end{tabular}
\end{scriptsize}
\end{center}
\caption{Numerical results for the relevant variables in the  stability quantification
for $N=1$ for the two families ($p=1$ and $p=2$). 
The case $p=2$ produces an extremely flat curve for $\Delta(e)$ close to the
origin. The estimation of $\Delta''(0)$ cannot be accurately determined but the relevant issue is the sign, which is negative (stable).  \label{tbl:n1}}
\end{table}
\begin{table}
\begin{center}
\begin{scriptsize}
\begin{tabular}{|c|c|c|c|c|}
\hline
 p & $\hat{E}$ & $\Delta''(0)$ &   $\mathcal{K}$ & $\mu_0$ \\ \hline \hline
1  &    &	 &	& \\ \hline	
2  &     &   &  &		\\ \hline
3  &     &   &  &		\\ \hline
4  &     &   &  &		\\ \hline
5  &     &   &  &		\\ \hline
6  &     &   &  &		\\ \hline
7  &     &   &  &		\\ \hline
8  &     &   &  &		\\ \hline
\end{tabular}
\end{scriptsize}
\end{center}
\caption{Numerical results for the relevant variables in the quantification
for $N=1$ for the eight families ($p=1$ through $p=8$). \label{tbl:n3}}
\end{table}

\vspace{0.5 cm}
\textbf{Remarks.} 
\begin{enumerate}
\item The values of $\Delta''(0)$ and  $\mathcal{K}$ have been computed by
polynomial interpolation. The value of $\Delta''(0)$ has been satisfactorily compared with the exact expression (\ref{derivada segunda-discri}).

 \item We have not presented results for $N=2$ because for even values of $N$
we have not been able to prove the eccentricity evenness of $\Delta(e)$ that explains
the vanishing of the odd derivatives of the discriminant function. Those results will be presented
elsewhere but they display a similar behaviour to the odd $N$ cases (i.e. 
all the even periodic solutions emanate as elliptic branches from the circular case).
\end{enumerate}

\section*{Acknowledgments}
The authors acknowledge fruitful discussions with Rafael Ortega on related topics. JGV's research has been financially supported by the Spanish Ministry of Economy through grant MTM2015-65608-P and Junta de Andaluc\'ia Excellence grant 
P12-FQM-1658.

\section*{Appendix 1}
In this appendix we present some inequalities used in this manuscript. Firstly we present an uniform bound of $\displaystyle{|a_{\xi,e}(t)|}$  with
\[
a_{\xi,e}(t):=\frac{r(t,e)^{2}-2z^{2}}{\big(z^2+r^2(t,e)\big)^{5/2}},
\]
with $z=z(t,\xi,0,e)$ for all $t\in [0,N\pi]$ with $(\xi,e)\in \Omega_{p}$. From (\ref{var}) and (\ref{cota kepler}) we get the following estimation  
\begin{equation*}
|a_{\xi,e}(t)|\leq \frac{2}{(z^{2}+r^{2}(t,e))^{3/2}}\leq\frac{16}{(1-e)^{3}}.
\end{equation*}
Thus
\[
\big|a_{\xi,e}(t)\big|\leq \sigma, \quad \sigma=\sigma(E):=\frac{16}{(1-e)^{3}},
\]
uniformly on $\Omega_{p}$. Applying the Mean Value Theorem uniformly in $\xi$ is not difficult to obtain the following inequalities

\[
\big|\partial_{e}r\big|\leq \frac{1}{2(1-e)}, \quad \left|\frac{\partial a_{\xi,e}}{\partial z}\right| \leq \frac{12\sigma}{1-e}, \quad \text{and} \quad \left|\frac{\partial a_{\xi,e}}{\partial r}\right| \leq \frac{12\sigma}{1-e}.
\]

For instance, for obtain the second inequality we proceed as follows

\begin{equation*}
\begin{split}
\left|\frac{\partial a_{\xi,e}}{\partial z}\right| &=\left| \frac{3z(3r^{2}-2z^{2})}{(z^{2}+r^{2})^{7/2}}\right|\\
&=\left|\frac{3z}{z^{2}+r^{2}}\left(\frac{2r^{2}}{(z^{2}+r^{2})^{5/2}}+\frac{r^{2}-2z^{2}}{(z^{2}+r^{2})^{5/2}}\right)\right|\\
&=\left|\frac{3z}{z^{2}+r^{2}}\left(\frac{2r^{2}}{(z^{2}+r^{2})^{5/2}}+a_{\xi,e}\right)\right|\\
&\leq \frac{6}{(z^{2}+r^{2})^{1/2}}\left(\frac{r^{2}}{(z^{2}+r^{2})^{5/2}}\right)+\frac{3|a_{\xi,e}|}{(z^{2}+r^{2})^{1/2}}\\
& \leq  \frac{6}{r^{4}}+\frac{3\sigma}{r}=\frac{12\sigma}{1-e},
\end{split}
\end{equation*}
and for the third one we have
\begin{equation*}
\begin{split}
\left|\frac{\partial a_{\xi,e}}{\partial r}\right| &=\left| \frac{3r(r^{2}-4z^{2})}{(z^{2}+r^{2})^{7/2}}\right|\\
&=\left|\frac{3r}{z^{2}+r^{2}}\left(\frac{r^{2}-2z^{2}}{(z^{2}+r^{2})^{5/2}}-\frac{2z^{2}}{(z^{2}+r^{2})^{5/2}}\right)\right|\\
&=\left|\frac{3r}{z^{2}+r^{2}}\left(a_{\xi,e}-\frac{2z^{2}}{(z^{2}+r^{2})^{5/2}}\right)\right|\\
&\leq \frac{3}{r}\left(|a_{\xi,e}|+\left|\frac{2z^{2}}{(z^{2}+r^{2})^{5/2}}\right|\right)\leq  \frac{3}{r}\Big(|a_{\xi,e}|+\frac{2}{r^{3}}\Big)\leq \frac{12\sigma}{1-e}.
\end{split}
\end{equation*}
On the other hand,
\[
\Big|\frac{\partial z}{\partial e}\Big|\leq \frac{3N\pi \sigma R_{e}^{2}}{2(1-e)},
\]
with $R_{e}$ is defined like in Lemma \ref{control hill} with $\lambda=e$. In fact, by the equations (\ref{norma de p})-(\ref{ecuacion integral}) in section  3.1 we have
\begin{equation*}
\Big|\frac{\partial z}{\partial e}\Big|\leq 2N\pi \left\| p\right\|_{\infty} R_{e}^{2} \leq \frac{24N\pi R_{e}^{2}}{(1-e)^{4}}\leq \frac{3N\pi \sigma R_{e}^{2}}{2(1-e)}.
\end{equation*}

\vspace{0.5 cm}
\noindent
Finally, from the chain rule and previous estimates we get 
\begin{equation*}
\begin{split}
\left|\frac{\partial a_{\xi,e}}{\partial e}\right|&=\left|\frac{\partial a_{\xi,e}}{\partial z} \frac{\partial z}{\partial e}+\frac{\partial a_{\xi,e}}{\partial r} \frac{\partial r}{\partial e}\right|\\
&\leq \frac{18 N\pi\sigma R_{e}^{2}}{(1-e)^{2}} + \frac{6\sigma}{(1-e)^{2}}=\frac{6\sigma\left(1+3N\pi\sigma R_{e}^{2}\right)}{(1-e)^{2}}.
\end{split}
\end{equation*}

\section*{Appendix 2}

The purpose of this appendix is to show how the distance of the primaries to their center of mass given by the function $r(t,e)$ can be formally extended for negative values of $e$ around $e=0.$ To this end, we recall that the function $r(\cdot,e)$ has minimal period $2\pi$ and satisfies
\[
r(t,e)=\frac{1}{2}\left[1-e\cos(u(t,e))\right],
\]
where $u(t,e)$ is the eccentric anomaly. In appropriate units $u(t,e)$ is a function of the time via the transcendental Kepler's equation 
\[
u-e\sin u=t.
\]

It is well know that $u(t,e)$ satisfies the following
\begin{equation}
u(t+2N\pi,e)=u(t,e)+2N\pi, \quad u(-t,e)=-u(t,e), \quad N\geq 1,
\end{equation}
for all $(t,e)\in \R\times [0,1[$. Hereinafter we assume $N\in\N$ odd. From the left equation follows directly
\begin{equation}\label{prop u}
u(t+N\pi,e)=u(t+\pi,e)+(N-1)\pi
\end{equation}

\noindent
Now, using the Lagrange formula for the local inversion of holomorphic functions we can represent $u(t,e)$ as an analytic function around $e=0$ in the following form
\begin{equation}\label{serie u}
u(t,e)=t+\sum_{k=1}^{\infty}c_{k}(t)\frac{e^{k}}{k!}, \quad \text{with}\quad c_{k}(t)=\frac{d^{k-1}}{dt^{k-1}}\sin^{k} (t)\quad k\geq 1.
\end{equation}

\noindent
The series (\ref{serie u}) converges for all $t\in \R$ and small values of $e$. Moreover, notice that the coefficients $c_{k}(t)$ satisfies 
\[
c_{k}(t+N\pi)=(-1)^{k}c_{k}(t),
\]
and therefore for negative values of $e$ we obtain
\begin{equation}
\begin{split}
u(t,-e)&:=t+\sum_{k=1}^{\infty}c_{k}(t)(-1)^{k}\frac{e^{k}}{k!}\\
u(t,-e)&=t+\sum_{k=1}^{\infty}c_{k}(t+N\pi)\frac{e^{k}}{k!}\\
&=u(t+N\pi,e)-N\pi.
\end{split}
\end{equation}

\noindent
This property of the function $u(t,e)$ lead us to obtain the following property over the function $r(t,e)$
\begin{equation*}
\begin{split}
r(t+N\pi,e)&=\frac{1}{2}\left[1-e\cos\left(u(t+\pi,e)+(N-1)\pi\right)\right]\\
&=\frac{1}{2}\left[1-e\cos\left(u(t,-e)+\pi\right)\right]\\
&=\frac{1}{2}\left[1+e\cos\left(u(t,-e)\right)\right]:=r(t,-e)
\end{split}
\end{equation*}


\begin{thebibliography}{9}


\bibitem{Alekseev II} V. M. Alekseev, Quasirandom dynamical systems II, Math. USSR Sbornik, 6 (1968), 505{-}560.

\bibitem{Amann} H. Amann, 
\newblock \emph{Ordinary Differential Equations An Introduction to Nonlinear Analysis}, 
\newblock Walter de Gruyter, 1990.

\bibitem{Belbruno} E. Belbruno, J. Llibre, and M. Oll´e, On the families of periodic orbits which bifurcate from the
circular Sitnikov motions, Celestial Mech. Dynam. Astronom., 60 (1994), 99-129.

\bibitem{Coddington} E.A. Coddington and N. Levinson,
\newblock \emph{Theory of Ordinary Differential Equations}, 
\newblock McGraw-Hill, New York, 1955.

\bibitem{Corbera} M. Corbera and J. Llibre, Periodic orbits of the Sitnikov problem via a Poincaré map, Celestial Mech.
Dynam. Astronom., 77 (2000), 273-303.

\bibitem{Corbera-Llibre} M. Corbera and J. Llibre, On symmetric periodic orbits of the elliptic Sitnikov problem via the analytic
continuation method, in Celestial Mechanics, Contemp. Math. 292, AMS, Providence, RI, 2002, 91-127.

\bibitem{Galan}
     \newblock J. Gal\'an-Vioque, F.J. Mu\~noz-Almaraz, E. Freire, and E. Freire,
     \newblock \emph{Continuation of periodic orbits in symmetric hamiltonian and conservative sytems},
     \newblock European Physical Journal Topics 223 (13), 2705-2722, 2014.

\bibitem{Jimenez} L. Jiménez-Lara and A. Escalona-Buendía, Symmetries and bifurcations in the Sitnikov problem,
Celestial Mech. Dynam. Astronom., 79 (2001) 97-117.

\bibitem{Krantz} S. Krantz, H. Parks, A primer of real analytic functions, Birkhauser 2002.

\bibitem{Llibre-Ortega} J. Llibre and R. Ortega, On the families of periodic orbits of the Sitnikov problem, SIAM J. Applied Dynamical
Systems, 7 (2008) 561-576.

\bibitem{Magnus-Winkler} W. Magnus and S. Winkler, Hill's Equation, Dover, New York, 1979.

\bibitem{Martinez-Chiralt} J. Martinez-Alfaro and C. Chiralt, Invariant rotational curves in Sitnikov's problem, Celestial Mech.
Dynam. Astronom., 55 (1993) 351-367.

\bibitem{Mathlouthi} S. Mathlouthi, Periodic orbits of the restricted three-body problem, Trans. Amer. Math. Soc., 350 (1998) 2265-2276.

\bibitem{Mennicken} R. Mennicken, On Ince's equation, Arch. Ration. Mech. Anal., 29 (1968) 144-160.

\bibitem{Moser}  J. Moser, Stable and random motions in dynamical systems, Annals of Math. Studies 77, Princeton
University Press, Princeton, NJ, 1973.


\bibitem{Munoz1}
     \newblock F.J. Mu\~noz-Almaraz, E. Freire, J. Gal\'an-Vioque, A. Vanderbauwhede,
     \newblock \emph{Continuation of normal doubly symmetric orbts in conservative reversible systems},
     \newblock Celestial Mechanics and Dynamical Astronomy 97 (1), 17-47, 2007.  

\bibitem{Munoz2}
     \newblock F.J. Mu\~noz-Almaraz, E. Freire, J. Gal\'an-Vioque, E.J. Doedel and A. Vanderbauwhede,
     \newblock \emph{Continuation of periodic orbits in conservative Hamiltonian sytems},
     \newblock Physica D 181:1-38, 2003.


\bibitem{Ortega-Rivera} 
     \newblock R. Ortega and A. Rivera,
     \newblock \emph{Global bifurcations from the center of mass in the Sitnikov problem},
     \newblock Discrete and Continuous Dynamical Systems. Series B, \textbf{14} (2010), 719-732. 


\bibitem{Perdios-Markellos} E. Perdios and V. V. Markellos, Stability and bifurcations of Sitnikov motions, Celestial
Mech.Dynam. Astronom., 42 (1988) 187-200.

\bibitem{Rabinowitz} P. Rabinowitz, Some global results for nonlinear eigenvalue problems, J. Funct. Anal., 7 (1971) 487-513.

\bibitem{Rivera}
\newblock A. Rivera,
\newblock \emph{Periodic Solutions in the Generalized Sitnikov N+1-body Problem},
\newblock SIAM J. Applied Dynamical Systems, \textbf{12} (2013) 1515-1540.

\bibitem{Rivera2}
\newblock A. Rivera,
\newblock \emph{Bifurcaci\'on de soluciones peri\'odicas en el problema de Sitnikov. PH.D. Thesis. Universidad de Granada. 2012}. 

\bibitem{Robinson} C. Robinson, Uniform subaharmonic orbit for Sitinikov problem, Discrete Contin. Dyn. Syst. Ser. S, 1 (2008) 647-652.

\bibitem{Tkhai} V. N. Tkhai, Periodic motions of a reversible second-order mechanical system. Application to the Sitnikov problem,
 J. of Applied Mathematics and Mechanics, 70 (2006) 734-753.

\bibitem{Whitley} D.Whitley, Discrete dynamical systems in dimensions one and two, Bull. London Math., 15 (1983) 177-217.


\end{thebibliography}
\end{document}